\documentclass[12pt,reqno]{amsart}
\usepackage[a4paper, hmargin={2.7cm,2.7cm},vmargin={3.3cm,3.3cm}]{geometry}
\usepackage[english]{babel}
\usepackage{color}
\usepackage[noadjust]{cite}
\usepackage{amsmath}
\usepackage{amssymb}
\usepackage{amsfonts}
\usepackage{amsthm}
\usepackage{enumerate}
\usepackage{amsthm}
\usepackage{dsfont}

\usepackage[textwidth=20mm]{todonotes}
\usepackage[hypertexnames=false]{hyperref}
\usepackage[nameinlink]{cleveref}
\usepackage{commath}

\usepackage{etoolbox}

\makeatletter
\providecommand*{\cupdot}{%
  \mathbin{%
    \mathpalette\@cupdot{}%
  }%
}
\newcommand*{\@cupdot}[2]{%
  \ooalign{%
    $\m@th#1\cup$\cr
    \sbox0{$#1\cup$}%
    \dimen@=\ht0 %
    \sbox0{$\m@th#1\cdot$}%
    \advance\dimen@ by -\ht0 %
    \dimen@=.5\dimen@
    \hidewidth\raise\dimen@\box0\hidewidth
  }%
}

\providecommand*{\bigcupdot}{%
  \mathop{%
    \vphantom{\bigcup}%
    \mathpalette\@bigcupdot{}%
  }%
}
\newcommand*{\@bigcupdot}[2]{%
  \ooalign{%
    $\m@th#1\bigcup$\cr
    \sbox0{$#1\bigcup$}%
    \dimen@=\ht0 %
    \advance\dimen@ by -\dp0 %
    \sbox0{\scalebox{2}{$\m@th#1\cdot$}}%
    \advance\dimen@ by -\ht0 %
    \dimen@=.5\dimen@
    \hidewidth\raise\dimen@\box0\hidewidth
  }%
}
\makeatother

\usepackage{amssymb}   
\usepackage{amsthm}    
\usepackage{thmtools}  
\usepackage{mathtools} 
\usepackage{mathrsfs}  
\usepackage{commath}
\usepackage{bbm}
\usepackage[textwidth=20mm]{todonotes}

\numberwithin{equation}{section}

\theoremstyle{plain}
\newtheorem{theorem}{Theorem}[section]
\newtheorem{lemma}[theorem]{Lemma}
\newtheorem{proposition}[theorem]{Proposition}
\newtheorem{corollary}[theorem]{Corollary}
\newtheorem{conjecture}[theorem]{Conjecture}

\theoremstyle{definition}

\theoremstyle{remark}
\newtheorem{remark}[theorem]{Remark}

\usepackage[all]{xy}

\newcommand{\N}{\mathbb{N}}
\newcommand{\Z}{\mathbb{Z}}

\newcommand\numberthis{\addtocounter{equation}{1}\tag{\theequation}}

\newcommand{\Hpi}{\mathcal{H}_{\pi}}

\DeclareMathOperator{\Rel}{rel}
\DeclareMathOperator{\covol}{covol}
\DeclareMathOperator{\loc}{loc}

\title{On the Calder\'on sum formula for wavelet systems}

\author{Ulrik Enstad}
\address{Department of Mathematics,
University of Oslo,
Moltke Moes vei 35,
0851 Oslo.}
\email{ubenstad@math.uio.no}

\author{Jordy Timo van Velthoven}
\address{Faculty of Mathematics,
University of Vienna, 
Oskar-Morgenstern-Platz 1,
1090 Vienna, Austria}
\email{jordy-timo.van-velthoven@univie.ac.at}



\begin{document}

\maketitle

\begin{abstract}
We show that the Calder\'on sum formula for orthonormal wavelet bases holds for arbitrary dilation and translation matrices under a mild condition on the wavelet function. This partially solves a conjecture by Bownik and Lemvig. 
\end{abstract}

\section{Introduction}
For $\psi \in L^2 (\mathbb{R}^d)$ and matrices $A, P \in \mathrm{GL}(d, \mathbb{R})$, the associated wavelet system is given by the collection of functions
\begin{align} \label{eq:discrete_wavelet}
\big\{ |\det(A)|^{-j/2} \psi (A^j \cdot -  Pk) \big \}_{j \in \mathbb{Z}, k \in \mathbb{Z}^d}.
\end{align}
In higher dimensions, a classical condition in the study of wavelet systems is to assume that the matrix $A$ preserves the integer lattice $\mathbb{Z}^d$, i.e., $A \mathbb{Z}^d \subseteq \mathbb{Z}^d$, and that $A$ is expansive, i.e., all its eigenvalues are strictly greater than one in modulus. Under such conditions, a full characterization of Parseval wavelet frames, or more generally dual wavelet frames, was obtained in \cite{chui2002characterization, bownik2000characterization, calogero2000characterization}, among others. In addition, the existence of wavelet bases for such dilations were shown in \cite{dai1997wavelet}. Both types of results extend classical results from dimension one to arbitrary dimensions. 

Beyond the case of expansive dilation matrices, the theory of wavelet systems in higher dimensions is far less complete. Nevertheless, the existence of wavelet bases for nonexpansive dilation matrices has been studied in \cite{speegle2003existence, bownik2017wavelets, ionascu2006simultaneous, wang2002wavelets} and culminated in the recent breakthrough \cite{bownik2021simultaneous} that characterizes the dilation matrices admitting wavelet sets. On the other hand, the aforementioned characterization of Parseval wavelet frames is currently only known for special dilations, such as amplifying dilations \cite{laugesen2002translational}, dilations expanding on a subspace \cite{hernandez2002unified, guo2006some} or dilations satisfying the lattice counting estimate \cite{bownik2017wavelets}. This has lead to the following conjecture \cite[Conjecture 1]{bownik2017wavelets} and open problem \cite[Problem 3.3]{bownik2020open}, which was already implicitly raised in \cite[p. 177]{speegle2003existence}.\footnote{The formulations in \cite{bownik2017wavelets, bownik2020open} are under the implicit assumption that $|\det(P)| = 1$.}

\begin{conjecture}[\cite{bownik2017wavelets, bownik2020open, speegle2003existence}] \label{conj:wavelet}
Let $A, P \in \mathrm{GL}(d, \mathbb{R})$ and $\psi \in L^2 (\mathbb{R}^d)$. Suppose that \[ \{|\det(A)|^{-j/2} \psi (A^j \cdot -  Pk) \}_{j \in \mathbb{Z}, k \in \mathbb{Z}^d}\]
is an orthonormal basis, or more generally a Parseval frame, for $L^2 (\mathbb{R}^d)$. Then the Calder\'on sum formula holds:
\begin{align} \label{eq:calderon}
\sum_{j \in \mathbb{Z}} |\widehat{\psi} ((A^t)^j \xi)|^2 = |\det(P)| \quad \text{for a.e.} \quad \xi \in \mathbb{R}^d.
\end{align}
\end{conjecture}

The Calder\'on sum formula \eqref{eq:calderon} is part of the aforementioned characterizion of Parseval wavelet frames known under additional assumptions on the dilation matrix. The upper bound for the Calder\'on sum is known to hold for any Bessel sequence with bound $1$ (cf. \cite[Proposition 4.1]{hernandez2002unified}).

In this paper, we present a new approach to \Cref{conj:wavelet} which allows us to prove the Calder\'on sum formula for arbitrary translation and dilation matrices under a mild condition on the wavelet function. Our approach is based on a relation between the frame properties of the discrete wavelet system \eqref{eq:discrete_wavelet} and the semi-continuous wavelet system whose elements are given by
\begin{align} \label{eq:continuous_wavelet}
\pi(x, A^j) \psi := |\det(A)|^{-j/2} \psi(A^{-j} ( \cdot - x)), \quad x \in \mathbb{R}^d, j \in \mathbb{Z}.
\end{align}
The action $\pi$ forms a unitary group representation of the semi-direct product group $G = \mathbb{R}^d \rtimes \langle A \rangle$ of $\mathbb{R}^d$ and the cyclic group $\langle A \rangle := \{ A^j : j \in \mathbb{Z} \}$ generated by $A \in \mathrm{GL}(d, \mathbb{R})$. Observe that the  wavelet system \eqref{eq:discrete_wavelet} corresponds to $\{\pi(A^j P k, A^j) \psi : j \in \mathbb{Z}, k \in \mathbb{Z}^d \}$. We will assume a mild condition on the wavelet function. Namely, we assume that $\psi \in L^2 (\mathbb{R}^d)$ is such that $\pi(\Lambda) \psi$ is a Bessel sequence in $L^2 (\mathbb{R}^d)$ for all relatively separated sets $\Lambda $ in $ G = \mathbb{R}^d \rtimes \langle A \rangle$; in notation, $\psi \in \mathcal{B}_{\pi}$. This is a common assumption in the study of frames in the orbit of a group representation, see, e.g., \cite{Gr08, FuGr07, enstad2025dynamical, fuehr2017density, caspers2023overcompleteness}.
We refer to Section \ref{sec:coefficient} for an alternative description of the space $\mathcal{B}_{\pi}$ and further properties.

Using the notation from the previous paragraph, our main result is the following:

\begin{theorem} \label{thm:calderon_intro}
Let $A, P \in \mathrm{GL}(d, \mathbb{R})$.
If $\psi \in \mathcal{B}_{\pi}$ and
$
 \{\pi(A^j P k, A^j) \psi \}_{ j \in \mathbb{Z}, k \in \mathbb{Z}^d } 
$
is a Parseval frame for $L^2 (\mathbb{R}^d)$, then $\{ \pi(x, A^j) \psi \}_{x \in \mathbb{R}^d, j \in \mathbb{Z}}$ is a tight frame for $L^2 (\mathbb{R}^d)$, and 
\[
\sum_{j \in \mathbb{Z}} |\widehat{\psi} ((A^t)^j \xi)|^2 = |\det (P)| \quad \text{for a.e.} \;\; \xi \in \mathbb{R}^d.
\]
\end{theorem} 

A combination of \Cref{thm:calderon_intro} with the known fact that  tight frames of the form $\{ \pi(x, A^j) \psi \}_{x \in \mathbb{R}^d, j \in \mathbb{Z}}$  exist only when $|\det(A)| \neq 1$ (cf. \cite{larson2006explicit, laugesen2002characterization}) yields the following consequence.

\begin{corollary}
    Let $A, P \in \mathrm{GL}(d, \mathbb{R})$.
If there exists a function $\psi \in \mathcal{B}_{\pi}$ such that
$
\big \{\pi(A^j P k, A^j) \psi \big\}_{ j \in \mathbb{Z}, k \in \mathbb{Z}^d } 
$
is a Parseval frame for $L^2 (\mathbb{R}^d)$, then  $|\det(A)| \neq 1$.
\end{corollary}

Our proof of \Cref{thm:calderon_intro} consists of showing lower and upper bounds on the norm of the continuous wavelet transform associated to the semi-continuous system \eqref{eq:continuous_wavelet} in terms of the frame bounds of the discrete wavelet frame \eqref{eq:discrete_wavelet}. In the case of a discrete Parseval wavelet frame \eqref{eq:discrete_wavelet}, the obtained \emph{identity} shows precisely that the semi-continuous system \eqref{eq:continuous_wavelet} is a continuous tight frame as asserted in \Cref{thm:calderon_intro}. 
Inequalities relating the continuous wavelet transform and the frame bounds of a discrete frame are classical for isotropic dilations, see, e.g., \cite{daubechies1992ten, chui1993inequalities, fuehr1996wavelet}. For general dilation matrices, only an upper bound of for such type of inequalities is known in general \cite{aniello2001discrete}.\footnote{Although the corresponding inequality involving the lower frame bound is also asserted in \cite{aniello2001discrete}, the claimed proof is incomplete; see also \Cref{rem:frame_bounds}.} We present a new approach towards obtaining both lower and upper bounds by interpreting such estimates as inequalities relating the norm of the wavelet transform, the frame bounds of a frame and the density of its point set. For this, we use a notion of density of a point set (the so-called \emph{covolume} with respect to an invariant measure on the hull of the given point set) recently introduced for unimodular groups in \cite{enstad2025dynamical}; see \Cref{sec:covolume} for precise details. Using this notion in the setting of nonunimodular groups, we derive inequalities (cf. \Cref{cor:quasi-lattice}) relating the covolume of the point set and the frame bounds of the frame that generalize similar inequalities for lattices in unimodular groups obtained in \cite{Rova22}; see also \cite{balan2006density} for such inequalities for Gabor systems.
As in \cite{Rova22, enstad2025dynamical}, the actual proof for such inequalities is based on the simple idea of periodizing the norm of the wavelet transform and applying the corresponding frame inequalities to each integrand. 

The above outlined proof method works naturally for general (projective) unitary representations of locally compact groups. As the auxiliary results might be of interest beyond wavelet systems, we present our results in this general setting. In particular, this allows us to obtain an extension of the main result of \cite{aniello2001discrete} for amenable groups beyond the setting of quasi-regular representations of semi-direct products (see \Cref{rem:frame_bounds}).

The paper is organized as follows. Section \ref{sec:pointsets} is devoted to background on point sets in locally compact groups. In particular, the notion of covolume is discussed here. Inequalities relating the frame bounds, covolume and norm of the wavelet transform are obtained in Section \ref{sec:frame}. Lastly, we show in Section \ref{sec:calderon} how Theorem \ref{thm:calderon_intro} can be derived from the general results of Section \ref{sec:frame}.

\section{Point sets in locally compact groups} \label{sec:pointsets}
Let $G$ be a second countable locally compact group. We fix a left and right Haar measure on $G$ denoted by $\lambda_G$ and $\rho_G$, respectively.

\subsection{Delone sets}
 Given a set $S \subseteq G$, we call a subset $\Lambda \subseteq G$ 
\begin{enumerate}
    \item \emph{$S$-separated} if $\#( \Lambda \cap g S) \leq 1$ for all $g \in G$, and
    \item \emph{$S$-dense} if $\#( \Lambda \cap g S ) \geq 1$ for all $g \in G$.
\end{enumerate}
A set $\Lambda$ is called \emph{separated} if it is $U$-separated for some nonempty open set $U \subseteq G$, and \emph{relatively dense} if it is $K$-dense for some compact set $K \subseteq G$. It is called a \emph{Delone set} if it is both separated and relatively dense. A set $\Lambda$ is called \emph{relatively separated} if it is a finite union of separated sets. Equivalently, $\Lambda$ is relatively separated if
\[
\Rel_Q (\Lambda) := \sup_{g \in G} \#(\Lambda \cap gQ ) < \infty
\]
for some (equivalently, all) nonempty open relatively compact sets $Q \subseteq G$.

A set $\Lambda$ is called a \emph{quasi-lattice} with complement $C$ if $G$ is the disjoint union of the sets $\lambda C$ for $\lambda \in \Lambda$. Note that this is equivalent to $\Lambda$ being both $C^{-1}$-separated and $C^{-1}$-dense. We will be interested in the case when $C$ is \emph{Jordan measurable}, that is, $\mu_G(\partial C) = 0$. 
Note that any lattice subgroup $\Lambda \subseteq G$ is a quasi-lattice. However, in contrast to lattice subgroups, a quasi-lattice does exist in every connected, simply connected solvable Lie group, cf. \cite[Proposition 5.10]{FuGr07}. 

\subsection{The hull of a point set}
Let $\mathcal{C}(G)$ be the set of all closed subsets of $G$. We equip $\mathcal{C}(G)$ with the so-called \emph{Chabauty--Fell topology}, which is the topology determined by the open neighborhood basis at each $C \in \mathcal{C}(G)$ consisting of the sets
\[
\mathcal{V}_{K, U} (C) := \big\{ C' \in \mathcal{C}(G) : C' \cap K \subseteq CU, \; \; C \cap K \subseteq C' U \big\},
\]
where $K$ and $U$ run through all compact subsets and open unit neighorhoods of $G$, respectively.  
We always consider $\mathcal{C}(G)$ as a topological space equipped with the Chabauty--Fell topology. The space $\mathcal{C}(G)$ is compact, metrizable and second-countable. 
We refer to \cite[Appendix A]{BjHaPo18} for further details on the Chabauty--Fell topology.

The group $G$ acts continuously on $\mathcal{C}(G)$ by means of the action
\[
\{ x C : x \in G \} \quad \text{for} \;\;\;  x \in G \quad \text{and} \quad C \in \mathcal{C}(G). 
\]
The closure of the orbit of a set $\Lambda \in \mathcal{C}(G)$ in $\mathcal{C}(G)$ is called the \emph{hull} of $\Lambda$ and is denoted by
\[
\Omega (\Lambda) := \overline{ \{x \Lambda : x \in G \}} \subseteq \mathcal{C}(G).
\]
The following lemma corresponds to \cite[Proposition 2.2]{enstad2025dynamical}.

\begin{lemma}\label{lem:relsep-reldense-passes-to-hull}
Let $U \subseteq G$ be open and let $K \subseteq G$ be compact. Then the following assertions hold:
\begin{enumerate}[(i)]
    \item If $\Lambda \subseteq G$ is $U$-separated, then so is every $\Gamma \in \Omega(\Lambda)$.
    \item If $\Lambda \subseteq G$ is $K$-dense, then so is every $\Gamma \in \Omega(\Lambda)$.
\end{enumerate}
\end{lemma}

Lastly, we state the following result, which is a special case of \cite[Proposition 3.3]{enstad2025dynamical}.

\begin{lemma} \label{lem:convergence_sets}
Let $U \subseteq G$ be an open set. Let $(\Lambda_n)_{n \in \mathbb{N}}$ be a sequence of $U$-separated sets that converges to a set $\Gamma \in \mathcal{C}(G)$. Let $K \subseteq G$ be a compact subset such that $\Gamma \cap \partial K = \emptyset$ and write $\Gamma \cap K = \{\gamma^{(1)}, ..., \gamma^{(k)} \}$. Then there exists $n_0 \in \mathbb{N}$ such that for every $n \geq n_0$ there exists $\mu_n^{(j)} \in \Lambda_n$, $1 \leq j \leq k$, satisfying 
\[
\Lambda_n \cap K = \{ \mu^{(1)}, ..., \mu^{(k)} \}
\]
and $\mu_n^{(j)} \to \gamma^{(j)}$ as $n \to \infty$ for each $1 \leq j \leq k$.
\end{lemma}

\subsection{Covolume of point sets} \label{sec:covolume}
Let $\Lambda \subseteq G$ be a Delone set. For $f \in C_c(G)$, we can form its periodization $\mathcal{P} f \in C_c(\Omega(\Lambda))$  by
\[ (\mathcal{P}f)(\Gamma) = \sum_{\gamma \in \Gamma} f(\gamma) , \quad \Gamma \in \Omega(\Lambda) .  \]
A Radon probability measure $\nu$ on $\Omega(\Lambda)$ is called \emph{invariant} if $\nu( g A) = \nu(A)$ for all $g \in G$ and measurable $A \subseteq \Omega(\Lambda)$. For such a measure $\nu$, the map
\[  f \mapsto \int_{\Omega(\Lambda)} \mathcal{P}(f) \dif{\nu}, \]
is a $G$-invariant linear functional on $C_c(G)$, hence is proportional to integrating against the left Haar measure $\mu_G$. Following \cite{enstad2025dynamical}, we call the inverse of this proportionality constant $\covol_\nu(\Lambda)$. That is,
\begin{equation}
    \int_{\Omega(\Lambda)} \sum_{\gamma \in \Gamma} f(\gamma) \dif{\nu(\Gamma)} = \covol_\nu(\Lambda)^{-1} \int_G f(x) \dif{\mu_G}(x). \label{eq:covolume}
\end{equation}
Note that $\covol_\nu(\Lambda)$ depends on both $\nu$ and $\mu_G$.

We mention that invariant measures $\nu$ on $\Omega (\Lambda)$ do not exist in general. However, they do exist for Delone sets $\Lambda \subseteq G$ in amenable groups $G$. Indeed, recall that $G$ is amenable if and only if there exists an invariant probability measure on every compact $G$-space, cf. \cite[Problem 14, p.90]{paterson1988amenability}. In addition, (unique) invariant measures exist on $\Omega(\Lambda)$ whenever $\Lambda$ is a lattice subgroup or a certain model set, cf. \cite[Theorem 3.4]{BjHaPo18}.

The following proposition provides useful bounds for the covolume of Delone sets.

\begin{proposition}\label{prop:covolume-quasi}
Let $\Lambda \subseteq G$ be a Delone set and let $\nu$ be an invariant probability measure on $\Omega(\Lambda)$. Then the following assertions hold:
\begin{enumerate}[(i)]
    \item If $\Lambda$ is $U$-discrete for some nonempty open set $U \subseteq G$, then $\covol_\nu(\Lambda) \geq \mu_G(U)$.
    \item If $\Lambda$ is $K$-dense for some compact set $K \subseteq G$, then $\covol_\nu(\Lambda) \leq \mu_G(K)$.
    \item If $\Lambda$ is a quasi-lattice with a relatively compact Jordan measurable complement $C$, then $\covol_\nu(\Lambda) = \rho_G (C)$.
\end{enumerate}
\end{proposition}
\begin{proof}
Applying the defining relation \eqref{eq:covolume} to the indicator function $\mathbbm{1}_S$ of a set $S \subseteq G$ of finite measure, we obtain
\begin{equation}
    \int_{\Omega(\Lambda)} \#( \Gamma \cap S ) \dif{\nu(\Gamma)} = \covol_\nu(\Lambda)^{-1} \mu_G (S). \label{eq:average-counting}
\end{equation}
We use this identity to prove the assertions.
\\~\\
(i) If $\Lambda$ is $U$-discrete, then by \Cref{lem:relsep-reldense-passes-to-hull}, we get that also $\#(\Gamma \cap U ) \leq 1$ for any $\Gamma \in \Omega(\Lambda)$. Hence, using that $\nu$ is a probability measure,  \eqref{eq:average-counting} immediately gives $1 \geq \covol_\nu(\Lambda)^{-1} \mu_G(U)$, as required.
\\~\\
(ii) If $\Lambda$ is $K$-dense, then \Cref{lem:relsep-reldense-passes-to-hull}, implies that also $\#(\Gamma \cap K ) \geq 1$ for any $\Gamma \in \Omega(\Lambda)$. Thus $1 \geq \covol_\nu(\Lambda)^{-1} \mu_G(K)$ follows from \eqref{eq:average-counting}.
\\~\\
(iii) The Jordan measurability of $C$ implies the Jordan measurability of $S = C^{-1}$ since inversion is continuous and the left and right Haar measures on $G$ are equivalent. Denote by $U = S^{\circ}$  and $K = \overline{S}$ the interior and closure of $S$, respectively. Note that $K$ is compact since $C$ is relatively compact. Jordan measurability gives $\mu_G(S) = \mu_G(U) = \mu_G(K)$. Since $U \subseteq S$ (resp.\ $K \supseteq S$) and $\Lambda$ is $S$-separated (resp.\ $S$-dense), it follows that $\Lambda$ must also be $U$-separated (resp.\ $K$-dense). Then (i) and (ii) give that
\[ \mu_G(S) = \mu_G(U) \leq \covol_\nu(\Lambda) \leq \mu_G(K) = \mu_G(S) . \]
Hence, $\covol_\nu(\Lambda) = \mu_G(S) = \mu_G(C^{-1}) = \rho_G (C)$, which settles the claim.
\end{proof}

Assertion (i) and (ii) of the above proposition were proved in \cite[Proposition 2.10]{enstad2025dynamical} using a slightly different but equivalent definition of covolume.

\section{Frame bounds, covolume and the norm of wavelet transform} \label{sec:frame}
In this section, we obtain inequalities relating the frame bounds of a frame, the covolume of its point set and the norm of the associated wavelet transform. 

\subsection{Wavelet transform} \label{sec:coefficient}
Let $(\pi, \Hpi)$ be a projective unitary representation of $G$ on a Hilbert space $\Hpi$, i.e., a strongly measurable map $\pi : G \to \mathcal{U}(\Hpi)$ satisfying $\pi(e_G) = I_{\Hpi}$ and
\[ \pi(g_1) \pi( g_2) = \sigma(g_1, g_2) \pi(g_1 g_2), \quad g_1, g_2 \in G ,\]
for some (measurable) function $\sigma : G \times G \to \mathbb{T}$.
For $\psi, f \in \Hpi$, we define the function
\[ C_\psi f(g) = \langle f, \pi(g) \psi \rangle, \quad g \in G. \]
 By \cite[Lemma 7.1, Theorem 7.5]{varadarajan1985geometry}, the absolute value $|C_{\psi} f| : G \to [0, \infty)$ is continuous for arbitrary $f, \psi \in \Hpi$.
We say that a vector $\psi \in \Hpi$ is \emph{admissible} if the associated \emph{wavelet transform} $C_{\psi} : \Hpi \to L^{\infty} (G)$ given by $f \mapsto C_{\psi} f$ is an isometry into $L^2 (G)$, that is, the orbit $\pi(G) \psi$ is a continuous Parseval frame for $\Hpi$. Throughout, we always consider $L^2 (G)$ with respect to the left Haar measure $\mu_G$.

For a fixed compact unit neighborhood $Q$, we define the associated local maximal function $MF$ of $F \in L^{\infty}_{\loc} (G)$  by
\[
(M F ) (g) = \sup_{q \in Q} |F (g q)|, \quad g \in G,
\]
and set $W(L^2) := \big\{ F \in L^{\infty}_{\loc}(G) : \|M F \|_{L^2} < \infty \big\}$. The space $W(L^2)$ forms a Banach space with respect to the norm $\| F \|_{W(L^2)} := \| M F \|_{L^2}$ and is independent of the choice of the defining neighborhood $Q$, see, e.g., \cite[Theorem 2.6, Corollary 4.2]{rauhut2007wiener}.

We define the space
\[
\mathcal{B}_{\pi} := \bigg\{ \psi \in \Hpi :  C_{\psi} f \in W(L^2), \quad \forall f \in \Hpi \bigg\}.
\]
The following lemma provides a simple sufficient condition for this space to be dense.

\begin{lemma} \label{lem:dense}
    Suppose that $(\pi, \Hpi)$ admits an admissible vector. Then the space $\mathcal{B}_{\pi}$ is norm dense in $\Hpi$.
\end{lemma}
\begin{proof}
    Let $\eta \in \Hpi$ be an admissible vector, so that $C_{\eta} (\Hpi)$ is a closed subspace of $L^2 (G)$. Consider the left regular $\sigma$-representation $(L^{\sigma}, L^2 (G))$ of $G$, defined by the action $(L^{\sigma} (h) F)(g) = \sigma(h, h^{-1} g) F(h^{-1} g)$ for $g, h \in G$ and $F \in L^2 (G)$. Since the space $\Hpi$ is $\pi$-invariant and 
    $
    C_{\eta} \pi(g) f = L^{\sigma} (g) C_\eta f
    $ for all $g \in G$ and $f \in \Hpi$, it follows that $C_{\eta} (\Hpi)$ is invariant under $L^{\sigma}$. Denote by $P$ the orthogonal projection from $L^2 (G)$ onto $C_{\eta} (\Hpi)$, which commutes with $L^{\sigma}(g)$ for every $g \in G$ as $C_{\eta} (\Hpi)$ is an $L^{\sigma}$-invariant subspace. Let $\psi \in C_{\eta}^{-1} P C_c (G)$, so that $C_{\eta} \psi = P F$ for some $F \in C_c (G)$. Then, given $f \in \Hpi$, a direct calculation shows that
    \begin{align*}
        |(C_{\psi} f )(g)| &= |\langle C_{\eta} f, C_{\eta} \pi(g) \psi \rangle| = |\langle C_{\eta} f, L^{\sigma} (g) C_{\eta}  \psi \rangle| = |\langle C_{\eta} f, P L^{\sigma} (g) F \rangle| \\
        &\leq (|C_{\eta} f| \ast |F^{\vee}|)(g),
    \end{align*}
    where $\ast$ denotes convolution on $G$ and $F^{\vee} (g) = F(g^{-1})$ for $g \in G$. Furthermore, this easily implies that
    $M C_{\psi} f \leq |C_{\eta} f| \ast M(F^{\vee})$. Note that $M(F^{\vee}) \in C_c (G)$, and thus $|C_{\eta} f | \ast M(F^{\vee}) \in L^2 (G)$, which yields that $M C_{\psi} f \in L^2 (G)$. Since $f \in \Hpi$ was arbitrary and $C_{\eta}^{-1} P C_c (G)$ is dense in $\Hpi$, the claim follows.
\end{proof}

Lastly, we rephrase two well-known properties of the space $W(L^2)$ in terms of  $\mathcal{B}_{\pi}$.

\begin{lemma}
Let $\psi \in \Hpi$. Then $\psi \in \mathcal{B}_{\pi}$ if and only if $\pi(\Lambda) \psi$ is a Bessel sequence in $\Hpi$ for every relatively separated set $\Lambda \subseteq G$. 
\end{lemma}
\begin{proof}
By \cite[Lemma 3.8]{feichtinger1989banach}, a continuous function $F : G \to \mathbb{C}$ is in $W(L^2)$ if and only if $(F(\lambda) )_{\lambda \in \Lambda} \in \ell^2(\Lambda)$ for every relatively separated set $\Lambda \subseteq G$. In particular, applying this to  the continuous function $|C_{\psi} f|$ for some $f \in \Hpi$, we get that $|C_{\psi} f| \in W(L^2)$ if and only if $(|C_{\psi} f(\lambda))|_{\lambda \in \Lambda} \in \ell^2 (\Lambda)$ for every relatively separated set $\Lambda \subseteq G$. The condition $(|C_{\psi} f(\lambda)|)_{\lambda \in \Lambda} \in \ell^2 (\Lambda)$ for all $f \in \Hpi$ is, by the uniform boundedness theorem, equivalent to $\pi(\Lambda) \psi$ being a Bessel sequence in $\Hpi$. Hence, the lemma follows.
\end{proof}

The following lemma is a simple consequence of \cite[Lemma 1]{Gr08}. 

\begin{lemma} \label{lem:sampling}
Let $\psi \in \mathcal{B}_{\pi}$, let $f \in \Hpi$ and let $Q \subseteq G$ be relatively compact. For every $\varepsilon > 0$ and $R > 0$, there exists a compact set $K \subseteq G$ such that for every relatively separated $\Lambda \subseteq G$ with $\Rel_Q (\Lambda) \leq R$, we have
\[ \sum_{\lambda \in \Lambda \cap K^c} |(C_{\psi} f) (\lambda)|^2  < \varepsilon.
\]
\end{lemma}

\subsection{Frames}
The following lemma corresponds to \cite[Theorem 3.9]{enstad2025dynamical}. As the paper \cite{enstad2025dynamical} has as a standing assumption that the group $G$ is unimodular, we provide a proof for the sake of completeness.

\begin{lemma}\label{lem:passes-to-hull}
Let $\psi \in \mathcal{H}_\pi$ and let $\Lambda \subseteq G$ be a separated set. Then the following hold:
\begin{enumerate}[(i)]
    \item If $\pi(\Lambda)\psi$ is a Bessel sequence in $\Hpi$, then $\pi(\Gamma)\psi$ is a Bessel sequence (with the same Bessel bound) for every $\Gamma \in \Omega(\Lambda)$.
    \item If $\pi(\Lambda) \psi$ is a frame for $\mathcal{H}$ with $\psi \in \mathcal{B}_{\pi}$, then $\pi( \Gamma) \psi$ is a frame (with the same frame bounds) for every $\Gamma \in \Omega(\Lambda)$.
\end{enumerate}
\end{lemma}

\begin{proof}
Let $\varepsilon > 0$, $f \in \Hpi$ and let $U \subseteq G$ be an open, relatively compact set such that $\Lambda$ is $U$-separated.
Let $\Gamma \in \Omega(\Lambda)$ and let $(\Lambda_n)_{n \in \mathbb{N}}$ be a sequence of left translates $\Lambda_n \subseteq G$ of $\Lambda$ such that $\Gamma = \lim_{n \to \infty} \Lambda_n$. Clearly, each set $\Lambda_n$, $n \in \mathbb{N}$, is $U$-separated and, by \Cref{lem:relsep-reldense-passes-to-hull}, also the set  $\Gamma$ is $U$-separated. 
\\~\\
(i) Let $\pi(\Lambda)\psi$ be a Bessel sequence with Bessel bound $C_2$. Then $\pi(\Lambda_n)\psi$ is a Bessel sequence with Bessel bound $C_2$ for all $n \in \N$. Let $K \subseteq G$ be a compact set. Enlarging $K$, we may assume that $\Gamma \cap \partial K = \emptyset$. Write $\Gamma \cap K= \{ \gamma^{(1)}, ..., \gamma^{(k)} \} $. Then, by \Cref{lem:convergence_sets}, there exists $n_0 \in \mathbb{N}$ such that $\Lambda_n \cap K = \{\lambda_n^{(1)}, ..., \lambda_n^{(k)} \}$ for all $n \geq n_0$ and $\lambda_n^{(j)} \to \gamma^{(j)}$ for all $1 \leq j \leq k$ as $n \to \infty$. This, combined with the continuity of the function $|C_{\psi} f|$ on $G$, yields that there exists $n_1 \geq n_0$ such that
\begin{align} \label{eq:passes-to-hull_2}
\bigg| \sum_{\gamma \in \Gamma \cap K} |(C_{\psi} f)(\gamma) |^2 - \sum_{\mu \in \Lambda_n \cap K} |(C_{\psi} f)(\mu) |^2 \bigg| \leq \frac{\varepsilon}{2}, \qquad n \geq n_1 .
\end{align}
Thus, since $\varepsilon > 0$ is arbitrary, we conclude that
\[ \sum_{\gamma \in \Gamma \cap K} | (C_\psi f)(\gamma)|^2 \leq \sum_{\mu \in \Lambda_n \cap K} |(C_\psi f)(\mu)|^2 \leq C_2 \| f \|^2 . \]
Since $f \in \Hpi$ and this holds for arbitrarily large compact sets $K \subseteq G$, we deduce that $\pi(\Gamma)\psi$ is a Bessel sequence with Bessel bound $C_2$.
\\~\\
(ii) Suppose $\pi(\Lambda)\psi$ is a frame and $\psi \in \mathcal{B}_{\pi}$. By \Cref{lem:sampling} there exists a compact set $K \subseteq G$ such that
\[
\sum_{\gamma \in \Gamma \cap K^c} |(C_{\psi} f)(\gamma) |^2 < \frac{\varepsilon}{4} \quad \text{and} \quad \sum_{\mu \in \Lambda_n \cap K^c} |(C_{\psi} f)(\mu) |^2 < \frac{\varepsilon}{4}
\]
for all $n \in \mathbb{N}$. Therefore,
\begin{align} \label{eq:passes-to-hull_1}
\bigg| \sum_{\gamma \in \Gamma \cap K^c} |(C_{\psi} f)(\gamma) |^2 - \sum_{\mu \in \Lambda_n \cap K^c} |(C_{\psi} f)(\mu) |^2 \bigg| \leq \frac{\varepsilon}{2}.
\end{align}
By enlarging $K$ if necessary, we may assume that \eqref{eq:passes-to-hull_1} holds for some compact $K \subseteq G$ with $\Gamma \cap \partial K = \emptyset$, see, e.g., the proof of \cite[Theorem 3.9]{enstad2025dynamical}. 

A combination of \eqref{eq:passes-to-hull_2} and \eqref{eq:passes-to-hull_1}  (with the same compact set $K \subseteq G$), we find $n_2 \in \N$ such that
\[
\bigg| \sum_{\gamma \in \Gamma} |(C_{\psi} f)(\gamma) |^2 - \sum_{\mu \in \Lambda_n \cap K^c} |(C_{\psi} f)(\mu) |^2 \bigg| \leq \varepsilon
\]
for all $n \geq n_2$. This easily implies the the desired conclusion.
\end{proof}

The following result is the main result of this section. 

\begin{theorem}\label{thm:frame-bounds-covolume}
Let $\Lambda \subseteq G$ be a Delone set and let $\nu$ be an invariant measure on $\Omega(\Lambda)$.
If $\pi(\Lambda) \psi$ is a Bessel sequence in $\Hpi$ for some $\psi \in \Hpi$ with Bessel bound $C_2 > 0$, then 
\[ \frac{ \| C_{\psi} f \|^2}{ \| f \|^2 \covol_\nu(\Lambda)} \leq C_2 \quad \text{for all} \quad f \in \Hpi . \]
In addition, if $\psi \in \mathcal{B}_{\pi}$ and $\pi(\Lambda) \psi$ also admits a lower frame bound $C_1 > 0$, then 
\[ \frac{ \| C_{\psi} f \|^2}{ \| f \|^2 \covol_\nu(\Lambda)} \geq C_1 \quad \text{for all} \quad f \in \Hpi . \]
\end{theorem}

\begin{proof} 
Let $f \in \Hpi$. 
By part (i) of \Cref{lem:passes-to-hull}, we have that
\[  \sum_{\gamma \in \Gamma} | \langle f, \pi(\gamma) \psi \rangle |^2 \leq C_2 \| f \|^2 \]
for any $\Gamma \in \Omega(\Lambda)$.
Integrating the term on the left-hand side with respect to $\nu$, we obtain
\begin{align*} \int_{\Omega(\Lambda)} \sum_{\gamma \in \Gamma} | \langle f, \pi(g) \psi \rangle |^2 \; \dif{\nu(\Gamma)} &= \covol_\nu(\Lambda)^{-1} \int_G |C_{\psi} f(g)|^2 \; d\mu_G (g) \\
&= \covol_\nu(\Lambda)^{-1} \| C_{\psi} f \|^2.  \numberthis \label{eq:covol_identity}
\end{align*}
Since $\nu$ is a probability measure, this gives
\[ \covol_\nu(\Lambda)^{-1} \| C_{\psi} f \|^2 \leq C_2 \| f \|^2, \]
from which we immediately can draw the first conclusion.

In addition, if $\pi(\Lambda) \psi$ with $\psi \in \mathcal{B}_{\pi}$ admits a lower frame bound $C_1 > 0$, then 
\[  \sum_{\gamma \in \Gamma} | \langle f, \pi(\gamma) \psi \rangle |^2 \geq C_1 \| f \|^2 \]
for any $\Gamma \in \Omega(\Lambda)$, by part (ii) of \Cref{lem:passes-to-hull}. Combined with the identity \eqref{eq:covol_identity}, this easily completes the proof.
\end{proof}

Lastly, we obtain the following consequence of \Cref{thm:frame-bounds-covolume} for amenable groups. 

\begin{corollary} \label{cor:quasi-lattice}
Let $G$ be amenable and let $U \subseteq G$  and $K \subseteq G$ be a nonempty open and compact set, respectively. Let $\Lambda \subseteq G$ be a $U$-separated, $K$-dense set and $\psi \in \Hpi$.

If $\pi(\Lambda)\psi$ is a Bessel sequence in $\Hpi$ with Bessel bound $C_2 > 0$, then
\[ \| C_{\psi} f \|^2 \leq C_2 \mu_G (K) \| f \|^2 \quad \text{for all} \quad f \in \Hpi. \]
 In addition, if $\psi \in \mathcal{B}_{\pi}$ and $\pi(\Lambda) \psi$ also admits a lower frame bound $C_1 > 0$, then
\[ \| C_{\psi} f \|^2 \geq C_1 \mu_G (U) \| f \|^2 \quad \text{for all} \quad f \in \Hpi. \]
In particular, if $\pi(\Lambda) \psi$ is frame with $\psi \in \mathcal{B}_{\pi}$ for some quasi-lattice $\Lambda \subseteq G$ with a relatively compact Jordan measurable complement $C$, then 
\[
C_1 \| f \|^2 \leq \frac{\|C_{\psi} f\|^2}{\rho_G (C)} \leq C_2 \| f \|^2
\]
for all $f \in \Hpi$.
\end{corollary}
\begin{proof}
Since $G$ is amenable, the compact $G$-space $\Omega(\Lambda)$  admits an invariant probability measure $\nu$. Applying \Cref{thm:frame-bounds-covolume}, we obtain
\[  \covol_\nu(\Lambda)^{-1} \| C_{\psi} f \|_2^2 \leq C_2 \| f \|^2 \]
for $f \in \Hpi$.
By \Cref{prop:covolume-quasi}, we have that $\covol_\mu(\Lambda) \leq \mu_G(K)$, from which the first claim follows. The additional claims are shown similarly.
\end{proof}

\begin{remark} \label{rem:frame_bounds}
 \Cref{cor:quasi-lattice} implies, in particular, that if $\pi(\Lambda) \psi$ is a Bessel sequence for some relatively dense set $\Lambda \subseteq G$, then $\pi$ must be square-integrable, in the sense that $C_{\psi} f \in L^2 (G)$ for all $f \in \Hpi$. 
For certain representations of semidirect product groups $G = \mathbb{R}^d \rtimes H$ of $\mathbb{R}^d$ and a matrix group $H \leq \mathrm{GL}(d, \mathbb{R})$, this was shown in \cite{aniello2001discrete}. For that setting, the upper bound provided by \Cref{cor:quasi-lattice} corresponds to \cite[Proposition 1]{aniello2001discrete}. Although the corresponding lower bound is asserted as \cite[Proposition 2]{aniello2001discrete} (even without any additional assumption on $\psi$), the argument provided for the lower bound in \cite{aniello2001discrete} (cf. \cite[Theorem 1]{aniello2001discrete}) is incomplete.
\end{remark}

\section{Calder\'on's condition for arbitrary dilations} \label{sec:calderon}
Throughout this section, we let $A \in \mathrm{GL}(d, \mathbb{R})$ and consider the semidirect product group $G = \mathbb{R}^d \rtimes \langle A \rangle$
of $\mathbb{R}^d$ and the cyclic group $\langle A \rangle := \{A^j : j \in \mathbb{Z}\}$, with group law
\[
(x, A^j) (y, A^k) = (x+A^j y, A^{j+k} ).
\]
Since $G$ is the semidirect product of two abelian (hence, amenable) locally compact groups, it is itself an amenable locally compact group, see, e.g., \cite[Proposition 0.15]{paterson1988amenability} and \cite[Problem 5, p.47]{paterson1988amenability}. 
The left and right Haar measure of $G$ are given by the measures $d\mu_G (x, A^j) = |\det(A)|^{-j} dx \mu_c(j)$ and $d\rho_G (x, A^{j}) = dx d\mu_c(j)$, respectively, where $dx$ denotes integration against Lebesgue measure and $\mu_c$ denotes counting measure. 

The following lemma provides a class of quasi-lattices in $G$.

\begin{lemma} \label{lem:quasilattice_affine}
For $A, P \in \mathrm{GL}(d, \mathbb{R})$, the set
\[ \Lambda := \big\{ (A^j P k, A^j) : j \in \mathbb{Z}, k \in \mathbb{Z}^d \big \}  \]
is a quasi-lattice in $G$ with Jordan measurable complement given by $C = P[0,1)^d \times \{ I_d \}$.
\end{lemma}

\begin{proof}
First, let $(A^j P k, A^j), (A^{j'} P k', A^{j'}) \in \Lambda$ be such that \[
 (A^j P k, A^j) C \cap (A^{j'} P k', A^{j'}) C \neq \emptyset. \] Then there exists $t,t' \in [0,1)^d$ such that $(A^j Pk, A^j)(Pt,I_d) = (A^{j'} P k', A^{j'})(Pt',I_d)$. This means that $A^j P k + A^jP t = A^{j'} P k' + A^{j'}Pt'$ and $A^j = A^{j'}$. The last forces $j=j'$, so the former gives $k+t=k'+t'$. Since $t,t' \in [0,1)^d$, we conclude that $t = t'$ and $k=k'$.

Second, let $(x,A^j) \in G$ be arbitrary. Choose $k \in \Z^d$ and $t \in [0,1)^d$  such that $P^{-1}A^{-j}x = k+t$. Then $(x,A^j) = (A^jP(k+t),A^j) = (A^jPk,A^j)(Pt,I_d)$. 

The above two properties show that $G = \bigcupdot_{\lambda \in \Lambda} \lambda Q$.
\end{proof}

We next consider the \emph{quasi-regular representation} of $G = \mathbb{R}^d \rtimes \langle A \rangle$ on $L^2 (\mathbb{R}^d)$, defined by
\[
\pi(x, A^j) f(t) = |\det(A)|^{-j/2} f(A^{-j} (t - x)), \quad t \in \mathbb{R}^d.
\]
It is readily verified that this action defines a unitary representation of $G$ on $L^2 (\mathbb{R}^d)$. By \cite[Section 3]{laugesen2002characterization} or \cite[Theorem 2]{larson2006explicit}, the representation $\pi$ admits an admissible vector whenever $|\det(A)| \neq 1$. In this case, the space $\mathcal{B}_{\pi}$ is norm dense in $L^2 (\mathbb{R}^d)$ by \Cref{lem:dense}.

We next prove the following more general version of \Cref{thm:calderon_intro}.

\begin{theorem} 
Let $A, P \in \mathrm{GL}(d, \mathbb{R})$.
If $\psi \in \mathcal{B}_{\pi}$ and
$
 \{\pi(A^j P k, A^j) \psi \}_{ j \in \mathbb{Z}, k \in \mathbb{Z}^d } 
$
is a frame for $L^2 (\mathbb{R}^d)$ with frame bounds $0<C_1\leq C_2 < \infty$, then $\{ \pi(x, A^j) \psi \}_{x \in \mathbb{R}^d, j \in \mathbb{Z}}$ is a continuous frame for $L^2 (\mathbb{R}^d)$ with frame bounds $C_1|\det(P)|$ and $ C_2|\det(P)|$, and 
\[
C_1  \leq \frac{1}{|\det(P)|}  \sum_{j \in \mathbb{Z}} |\widehat{\psi} ((A^t)^j \xi)|^2 \leq C_2  \quad \text{for a.e.} \;\; \xi \in \mathbb{R}^d.
\]
\end{theorem}
\begin{proof}
By \Cref{lem:quasilattice_affine}, the set $\Lambda = \{ (A^j Pk, A^j) : j \in \mathbb{Z}, k \in \mathbb{Z}^d \}$ is a quasi-lattice in $G$ with complement $C = P [0, 1)^d \times \{I_d\}$, so that $\rho_G (C) = |\det(P)|$. 
Hence, if the system $\pi(\Lambda) \psi$
is a frame for $L^2 (\mathbb{R}^d)$ with $\psi \in \mathcal{B}_{\pi}$, then \Cref{cor:quasi-lattice} implies that
\begin{align*} \label{eq:multiple_isometry}
C_1 |\det (P)| \leq \| C_{\psi} f \|^2  \leq C_2 |\det (P)| \| f \|^2
\end{align*}
for all $f \in L^2 (\mathbb{R}^d)$, which shows that 
$\{ \pi(x, A^j) \psi \}_{x \in \mathbb{R}^d, j \in \mathbb{Z}}$ is a frame with the claimed bounds.
A standard calculation (see, e.g., \cite{laugesen2002characterization, fuehr2002continuous}) next shows that
\begin{align*}
\| C_{\psi} f \|^2 &= \int_{G} | \langle \widehat{f}, \widehat{\pi(x, A^j) \psi} \rangle |^2 \; d\mu_G(x,A^j) \\
 &= \int_G \bigg| \int_{\mathbb{R}^d} \widehat{f}(\xi) |\det (A) |^{j/2} e^{- 2\pi i \xi \cdot x} \overline{\widehat{\psi}((A^t)^{j} \xi)} \; d\xi \bigg|^2 \; d\mu_G(x, A^j)  \\
 &= \sum_{j \in \mathbb{Z}} \int_{\mathbb{R}^d} \bigg| \int_{\mathbb{R}^d} \widehat{f}(\xi)   \overline{\widehat{\psi}((A^t)^{j} \xi)} e^{- 2\pi i \xi \cdot x} \; d\xi \bigg|^2 \; dx \\
 &= \sum_{j \in \mathbb{Z}} \int_{\mathbb{R}^d} \big|\widehat{f} (\xi) \widehat{\psi} ((A^t)^j \xi)\big|^2 \; d\xi \\
 &= \int_{\mathbb{R}^d} |\widehat{f} (\xi)|^2 \sum_{j \in \mathbb{Z}} |\widehat{\psi} ((A^t)^j \xi)|^2 \; d\xi, 
\end{align*}
where the penultimate step used Plancherel's formula. Hence,
\[
C_1 |\det(P)| \| \widehat{f} \|^2 \leq \int_{\mathbb{R}^d} |\widehat{f} (\xi)|^2 \sum_{j \in \mathbb{Z}} |\widehat{\psi} ((A^t)^j \xi)|^2 \; d\xi
\leq C_2 |\det(P)| \| \widehat{f} \|^2.
\]
Since this holds for arbitrary $f \in L^2 (\mathbb{R}^d)$, it follows that
\[
C_1 |\det(P)| \leq \sum_{j \in \mathbb{Z}} |\widehat{\psi} ((A^t)^j \xi)|^2 \leq C_2 |\det(P)|
\]
for a.e. $\xi \in \mathbb{R}^d$.
\end{proof}

\section*{Acknowledgements}
The second named author thanks Felix Voigtlaender for helpful discussions on the relation between discrete and continuous Parseval wavelet frames that motivated the proof approach used in this paper. For J.v.V., this research was funded in whole or in part by the Austrian Science Fund (FWF): 10.55776/PAT2545623. For open access purposes, the author has applied a CC BY public copyright license to any author-accepted manuscript version arising from this submission.

\bibliographystyle{abbrv}
\bibliography{bibl}

\end{document}